\documentclass[final]{siamltex}

\usepackage{graphicx} 
\usepackage{url}      
\usepackage{amsmath}  
\usepackage{amsfonts}
\usepackage{amssymb}
\usepackage{bm}
\usepackage{color}
\usepackage{listings}
\usepackage{url}
\usepackage{snapshot}
\usepackage{hyperref}

\newcommand{\R}{{\mathbb R}}
\newcommand{\C}{{\mathbb C}}

\newcommand{\dsumab}{\sum_{k_x,k_y=-n}^{n-1}} 
 
\newcommand{\dsumq}{\sum_{q_x,q_y=-n}^{n-1}}

\newtheorem{thm}{Theorem}[section]
\newtheorem{lem}[thm]{Lemma}

\pdfoutput=1

\title{Nonlinear Analysis of the Solutions of the Hasegawa--Wakatani Equations}
\author{Linda Stals
\thanks{Centre for Mathematics and its Applications,
Mathematical Sciences Institute,
Australian National University, Canberra, ACT 0200, Australia}}
\begin{document}
\maketitle

\begin{abstract}
The Hasegawa--Wakatani models are used in the study of confinement of hot plasmas with externally imposed magnetic fields.

The nonlinear terms in the Hasegawa--Wakatani models complicate the analysis of the system as they propagate local changes across the entire system. Centre manifold analysis allows us to project down onto much smaller systems that are more easily analysed. Qualitative information about the behaviour of the reduced system, such as whether it is stable or unstable, can be used to predict the behaviour of the original full system.
We show how the simple structure of the linear part of the Hasegawa--Wakatani equations can be used to define these projection operators. 

The centre manifold analysis will be used on a few examples to highlight certain properties of the Hasegawa--Wakatani models. 
\end{abstract}



\section{Introduction}

A number of papers focus on interpreting the numerical simulations of the Hasegawa--Wakatani equations \cite{Biskmap1994,Camargo1995, Korsholm1999, Pedersen:1996}, but few have tried to analyse the system's behaviour and those that have are applicable under limited conditions \cite{Camargo:1998, NumataBallDewar_07a, Pedersen:1996b, stals2008}. In this paper we develop a centre manifold analysis that predicts and explains many interesting aspects of the system's behaviour, including examples of unstable solutions. The particular structure of the Hasegawa--Wakatani system allows us to explicitly write down an approximation to the centre manifold, without the need of any expensive computation, and thus we can study the behaviour of the system over a wide range of parameters. See Section \ref{sec:analysis}.
 
The author has written a code to solve the Hasegawa--Wakatani equations and the initial motivation behind the work presented in this paper was the validation of that code. Sections \ref{sec:eg} and \ref{sec:eg_zero} relate the numerical simulations of the Hasegawa--Wakatani equations to the centre manifold analysis.

\section{The Hasegawa--Wakatani model}\label{sec:hw}

The Hasegawa--Wakatani model \cite{Hasegawa:1983} was designed to extend
the one field Hasegawa--Mima model \cite{Hasegawa:1978}. The equations
in a two-dimensional domain couple the flow field given by the electrostatic potential $\phi$ with the density $\rho$ by
 \begin{eqnarray}
 \frac{\partial}{\partial t} \nabla^2 \phi + [\phi, \nabla^2 \phi] & = &
  \alpha (\phi-\rho) -(-1)^p \beta_{\phi}  \nabla^{2(p+1)}\phi \label{eqn:hw1}\\
 \frac{\partial}{\partial t} \rho + [\phi, \rho] & = & \alpha (\phi-\rho) -\kappa
  \frac{\partial \phi}{\partial y} -  (-1)^p \beta_{\rho} \nabla^{2p} \rho \label{eqn:hw2}
 \end{eqnarray}
 where  $\phi = \phi(x, y, t)$, $\rho = \rho(x, y, t)$, $\nabla^{2} = \partial^2/\partial x^2+\partial^2/\partial y^2$ is the two-dimensional Laplacian and $p=1,2$ is the order of the dissipation operator. The Poisson bracket $[.,.]$ defined by
 \[
 [f,g] = \frac{\partial f}{\partial x} \frac{\partial g}{\partial y}  -
 \frac{\partial f}{\partial y} \frac{\partial g}{\partial x},
 \]
gives the convective derivative.  The parameters, $\alpha$,
 $\beta_{\phi}$, $\beta_{\rho}$ and $\kappa$ are usually taken to be 
 non-negative, but to simplify the analysis we will assume  $\alpha$,
 $\beta_{\phi}$, $\beta_{\rho}$ are all positive. The coefficients $\beta_\phi$ and $ \beta_{\rho}$  denote 
viscosity and diffusion  coefficients respectively. These terms, often called 
hyper-diffusions, may be small but are necessary to damp fluctuation energy that reaches the smallest scales. 
Physically, the $\kappa$ term feeds energy to the
 system, while the energy is dissipated by the parallel resistivity
 ($\alpha$) and  hyper-diffusions.  We refer the
reader to
\cite{Camargo:1998,Hasegawa:1983,NumataBallDewar_07a,Pedersen:1996} for
more discussions on the physical interpretation of the system.

\subsection{Spectral discretisation of the Hasegawa--Wakatani
 system}\label{sec:fs}

It is appropriate to use a Fourier spectral method to discretise
Equations \eqref{eqn:hw1} and \eqref{eqn:hw2} since much of the
essential physics of interest is contained in the energy distributions
over scales of the motions.

We assume that $\phi$ and $\rho$ are periodic over the domain 
$\Omega = [-\pi, \pi] \times [-\pi, \pi]$. Following the structure given in \cite{Geveci:1989} define
\begin{eqnarray*}
H^{p+1}(\Omega)& =&  \left \{u: u= \sum_{k_x,k_y=-\infty}^{\infty}  U_{\bm{k}}\,  \omega_x^{k_x}\,\omega_y^{k_y},\, U_{(0,0)} = 0, \, U_{(-k_x, -k_y)} = \overline{U}_{(k_x, k_y)}, \right.\\
&& \quad \left . \sum_{k_x,k_y=-\infty}^{\infty}  \left(\bm{k}^{2(p+1)}|U_{\bm{k}}|\right)^2 < \infty\right \},
\end{eqnarray*}
where  $\omega_x = e^{i x}$, $\omega_y = e^{i y}$, $\bm{k} = (k_x, k_y)$, $\bm{k}^{2q} = (k_x^2+k_y^2)^q$ and  $k_x$, $k_y$ are integers.
We assume that $\phi, \rho \in H^{p+1}$. Given $m = 2n$ for some positive integer $n$, let the corresponding discrete space $H_m^{p+1} \subset H^{p+1}$ be defined as
\begin{equation}\label{eqn:space}
H^{p+1}_m(\Omega) = \left \{u: u \in H^{p+1}(\Omega), \,u= \dsumab  U_{\bm{k}}\,  \omega_x^{k_x}\,\omega_y^{k_y} \right \}.
\end{equation}

Consequently $\phi_m \in H_m$ and $\rho_m \in H_m$ may be written as 
\begin{eqnarray}\label{eqn:f1}
\phi_m(x, y) & = & \dsumab \Phi_{\bm{k}}\, \omega_x^{k_x}\,\omega_y^{k_y} ,\\
\label{eqn:f2}
\rho_m(x, y) & = &\dsumab R_{\bm{k}}\, \omega_x^{k_x}\,\omega_y^{k_y}.
\end{eqnarray}

Substituting expansions \eqref{eqn:f1} and \eqref{eqn:f2}
into the Hasegawa--Wakatani equations \eqref{eqn:hw1} and \eqref{eqn:hw2}, making use of the differentiation properties of the Fourier transforms 
and collecting like terms gives 

\begin{equation}\label{eqn:ode-1}
 \begin{split}
-\bm{k}^2 \frac{\partial}{\partial t} \Phi_{\bm{k}} + (N(\bm{\Phi},
  D_N\bm{\Phi}))_{\bm{k}} & = \alpha\left(\Phi_{\bm{k}} - R_{\bm{k}} \right)+\bm{k}^{2(p+1)} \beta_{\phi}\Phi_{\bm{k}},\\
\frac{\partial}{\partial t} R_{\bm{k}} + (N(\bm{\Phi}, \bm{R}))_{\bm{k}} & = \alpha
  \left(\Phi_{\bm{k}} - R_{\bm{k}}\right) - ik_y \kappa\, \Phi_{\bm{k}} -
  \bm{k}^{2p} \beta_{\rho} R_{\bm{k}},
 \end{split}
\end{equation}
and $D_N = \diag(\bm{k}^2)$.
The nonlinear terms $N(\bm{\Phi}, D_N\bm{\Phi})$ and $N(\bm{\Phi}, \bm{R})$ are given
by the convolution sums;
\begin{equation} \label{eqn:nonlin_1}
\left(N(\bm{\Phi}, D_N\bm{\Phi})\right)_{\bm{k}}
=  \dsumq (k_x q_y-q_x k_y) \left((k_x-q_x)^2+(k_y-q_y)^2\right)\Phi_{\bm{q}} \Phi_{\bm{k}-\bm{q}},
\end{equation}
and
\begin{equation}\label{eqn:nonlin_2}
\left(N(\bm{\Phi}, \bm{R})\right)_{\bm{k}} 
=   \dsumq (k_x q_y-q_x k_y) \Phi_{\bm{q}} R_{\bm{k}-\bm{q}}. 
\end{equation}
It is convenient to write down the nonlinear term as a convolution sum for the theoretical analysis, but in practise these terms are evaluated using fast Fourier transformations.

To fit the standard definition of an ordinary differential equation we would like to divide both sides of  Equation \eqref{eqn:ode-1} by $-\bm{k}^2$, but $\bm{k}^2$ is zero when $\bm{k} = (0,0)$.  In Fourier space the $(0,0)$ mode is equivalent to the $(2n, 2n)$ mode, which suggest the following definition,
\[
\bm{k}^2_{+} = \left \{ \begin{array}{cc}
                           \bm{k}^2 & \mbox{ if } \bm{k}^2 \ne 0\\
                           8n^2 & \mbox{ if } \bm{k}^2 = 0
			\end{array} \right ..
\] 

 The Hasegawa--Wakatani Equations in Fourier space are
\begin{eqnarray}
\frac{\partial}{\partial t} \Phi_{\bm{k}} & = &\frac{(N(\bm{\Phi}, D_N\bm{\Phi}))_{\bm{k}}}{\bm{k}^2_+}  - \frac{\alpha\left(\Phi_{\bm{k}} - R_{\bm{k}} \right)}{\bm{k}^2_+} -  \beta_{\phi} \bm{k}^{2p} \Phi_{\bm{k}} \label{eqn:fhw1},\\
\frac{\partial}{\partial t} R_{\bm{k}} & = & - (N(\bm{\Phi}, \bm{R}))_{\bm{k}} + \alpha \left(\Phi_{\bm{k}} - R_{\bm{k}}\right) - ik_y \kappa \Phi_{\bm{k}} - \beta_{\rho} \bm{k}^{2p}  R_{\bm{k}},\label{eqn:fhw2}
\end{eqnarray}
for $-n \le k_x, k_y < n$.

From the definition of $H_m$, see \eqref{eqn:space}, we require $\Phi_{\bm{0}}=R_{\bm{0}}=0$ for all values of $t$. According to Equations \eqref{eqn:fhw1} and \eqref{eqn:fhw2}, if the initial values of ${\Phi}_{\bm{0}}$ and ${R}_{\bm{0}}$ are 0 they remain 0.

\subsection{The numerical solution of the Hasegawa--Wakatani equations}

 The system of ordinary differential equations defined by \eqref{eqn:fhw1} and \eqref{eqn:fhw2} are stiff and appropriate numerical schemes must be used to extract long term behaviour patterns. The author has studied the applicability of several numerical schemes for the solution of the Hasegawa--Wakatani equations, see \cite{stals:961}, and has written a Fortran 90 code that solves the equations using a fourth order implicit-explicit variable time step BDF scheme \cite{ Soderlind2003,wang:2005, Wang2008}. The results reported in Section \ref{sec:eg} were obtained through the use of this code.  High frequency components are damped by using a 2/3 dealising scheme similar to the one described by Pedersen et. al. \cite{Pedersen:1996}. Pedersen et. al. also use a semi-implicit scheme.

Another code referenced by a number of papers \cite{Biskmap1994, Camargo1995, Camargo:1998} has been developed by Scott \cite{Scott:2006,Scott:2007}. In \cite{Scott:2007} Scott argues that implicit methods are not appropriate for the Hasegawa--Wakatani equations. Our analysis in \cite{stals:961} shows that the Hasegawa--Wakatani equations are stiff and the bound on the time step size required to ensure stability makes the use of explicit methods impractical. We have found the implicit-explicit BDF method to be  reliable and efficient.

\section{Eigendecomposition}\label{sec:linear}

The idea of the centre manifold analysis is to project the large system defined by Equations \eqref{eqn:fhw1} and \eqref{eqn:fhw2} onto a much smaller system that retains much of the qualitative type of behaviour of the original system. In order to find such a projection we firstly need to look at the eigenvalues and eigenvectors of the linear part the Hasegawa--Wakatani equations.

Rewrite Equations \eqref{eqn:fhw1} and \eqref{eqn:fhw2} as
\begin{equation}\label{eqn:lin1}
\frac{\partial }{\partial t} \begin{bmatrix} \bm{\Phi}\\ \bm{R}\end{bmatrix}
 = L \begin{bmatrix} \bm{\Phi}\\ \bm{R}\end{bmatrix} + F\left(\bm{\Phi}, \bm{R}\right),
\end{equation}
where
\[
L = \begin{bmatrix} A & B\\ C & D \end{bmatrix},
\]
\[
A = \diag\left(-\frac{\alpha}{\bm{k}^2_+} - \beta_{\phi}\bm{k}^{2p}\right), \quad B = \diag\left(\frac{\alpha}{\bm{k}^2_+}\right),
\]
\[
 C = \diag\left(\alpha - ik_y\kappa\right), \quad  D =\diag\left(-\alpha- \beta_{\rho}\bm{k}^{2p}\right),
\]
and
\[
 F\left(\bm{\Phi}, \bm{R}\right) 
= \begin{bmatrix} (\diag(\bm{k}^2_+))^{-1} & 0\\ 0 & I \end{bmatrix}\begin{bmatrix} N(\bm{\Phi}, D_N\bm{\Phi}) \\  - N(\bm{\Phi}, \bm{R})\end{bmatrix}.
\]
Recall $-n \le k_x, k_y < n$. So $\bm{\Phi} \in \C^{m^2}$, $\bm{R} \in \C^{m^2}$ and $L \in \C^{2m^2,2m^2}$.
\subsection{Eigenvalues of linear system} \label{sec:eigenvalue}
A linear approximation to \eqref{eqn:fhw1} and \eqref{eqn:fhw2} about the point $(0,0)$ is given by
\[
\frac{\partial }{\partial t} \begin{bmatrix} \bm{\Phi}\\ \bm{R}\end{bmatrix}
 = L \begin{bmatrix} \bm{\Phi}\\ \bm{R}\end{bmatrix}.
\]
Given the block structure of the $L$ matrix it is possible to explicitly write down its eigenvalues and eigenvectors. Such information is crucial to the centre manifold analysis as it shows which modes influence the long term behaviour of the system. 


Suppose that $\lambda$ is not an eigenvalue of $A$ and consider
\begin{eqnarray*}
|L - \lambda I| &=& \left|\begin{array}{cc}
                          A - \lambda I & B\\
                          C & D-\lambda I
                          \end{array} \right|\\
&=& \left|\begin{array}{cc}
                          A - \lambda I & B\\
                          0 & (D-\lambda I)-C(A-\lambda I)^{-1}B
                          \end{array} \right|.
\end{eqnarray*}
$I \in \R^{m^2, m^2}$ is an identity matrix. The above determinant will be zero when
\begin{equation}\label{eqn:eig}
(A-\lambda I)(D-\lambda I)-CB = \lambda^2I - \lambda(A+D)+AD-CB = 0.
\end{equation}
Hence the $2m^2$, not necessarily distinct, eigenvalues are given by
\begin{equation}\label{eqn:lambda}
\lambda^{\pm}_{\bm{k}} = \frac{(A+D)_{\bm{k}}\pm \sqrt{\left((A+D)_{\bm{k}}\right)^2-4\,(AD-CB)_{\bm{k}}}}{2}.
\end{equation}

The following lemma confirms that the matrix $A-\lambda I$ is non-singular.

\begin{lem}\label{lem:eigA}
Eigenvalues of $L$ are not eigenvalues of $A$. 
\end{lem}
\begin{proof}
Suppose one of the eigenvalues of $L$, say $\lambda$, was an eigenvalue of $A$. Then, as $A$ is a diagonal matrix, $A_{\bm{k}} - \lambda = 0$ for some $\bm{k}$. Substituting this into \eqref{eqn:eig} gives $CB_{\bm{k}} = 0$, which intern implies $\alpha = 0$. A contradiction on the assumption that $\alpha > 0$.
\end{proof}

For the stability analysis we are interested in the sign of the real part of the eigenvalues. Firstly note that
\[
(A+D)_{\bm{k}} = -\left(\frac{\alpha}{\bm{k}^2_+}+\alpha+(\beta_{\phi}+\beta_{\rho})\bm{k}^{2p}\right) \in \R < 0,
\]
so $\Re(\lambda^-_{\bm{k}}) < 0$ for all $\bm{k}$.
For terms under the square root we see that
\begin{eqnarray*}
(AD-CB)_{\bm{k}} &=& \frac{\alpha+\beta_{\phi}\bm{k}^{2p}\bm{k}^2_+}{\bm{k}^{2}_+} \left(\alpha+\beta_{\rho}\bm{k}^{2p}\right)-\frac{\alpha}{\bm{k}^2_+}\left(\alpha-ik_y\kappa\right)\\
&=& \frac{1}{\bm{k}^2_+}\left[\alpha\beta_{\rho}\bm{k}^{2p}+\alpha\beta_{\phi}\bm{k}^{2(p+1)}+\beta_{\phi}\beta_{\rho}\bm{k}^{2(2p+1)}+i\alpha k_y\kappa\right] 
\end{eqnarray*}
Consequently, if $\kappa$ is large enough, $\Re(\lambda^+_{\bm{k}}) > 0$ for small values of $\bm{k}^2$ and $\Re(\lambda^+_{\bm{k}}) < 0$ for large values of $\bm{k}^2$. 

When $\bm{k} = (0,0)$, $\lambda_{\bm{k}}^{+} = 0$. We always assume that the initial values of $\bm{\Phi}_{(0,0)}$ and $\bm{R}_{(0,0)}$ are 0 and remain 0 over time, avoiding the issue of non-uniqueness arising from the 0 eigenvalue.

For large values of $\bm{k}^2$,
\[
\lambda^{\pm}_{\bm{k}} \approx \frac{-(\beta_{\phi}+\beta_{\rho})\bm{k}^{2p}\pm |\beta_{\phi}-\beta_{\rho}|\bm{k}^{2p}}{2}.
\]
In which case the eigenvalues are approximately $-\beta_{\rho}\bm{k}^{2p}$ or $-\beta_{\phi}\bm{k}^{2p}$, suggesting that the high frequency components of the solution  quickly approach zero.

Studies of the linear part of the Hasegawa-Wakatani equation have already been carried out by several authors, see \cite{Camargo1995, Camargo:1998, NumataBallDewar_07a, stals2008}, although in these references the analysis was not expressed explicitly in terms of eigenvalues. 

\subsection{Eigenvectors}\label{sec:eigenvector}
 Once again, due to the block diagonal structure of the $L$ matrix we can explicitly write down the eigenvectors. 

Consider one of the eigenvalues $\lambda_{\bm{k}}^{+}$ (or $\lambda_{\bm{k}}^-$) and let $\bm{w} \in \R^{m^2}$ be defined by
\[
w_{\bf{j}} = \left\{\begin{array}{cl}
                    1, & \mbox{ if } \bm{j} = \bm{k} \\
                    0, & \mbox{otherwise}
                    \end{array} \right..
\]
Now let $\bm{u} \in \C^{m^2}$ be given by
\[
\bm{u}^{+} = -(A-\lambda_{\bm{k}}^{+})^{-1}B \bm{w}.
\]
(or
\[
\bm{u}^{-} = -(A-\lambda_{\bm{k}}^{-})^{-1}B \bm{w}.)
\]

By Lemma \ref{lem:eigA} $(A-\lambda_{\bm{k}}^{\pm}I)$ is nonsingular and as it is a diagonal matrix the inverse can be easily found. The eigenvector corresponding to the eigenvalue $\lambda^{\pm}_{\bm{k}}$ is $\begin{bmatrix} \bm{u}^{\pm} & \bm{w}\end{bmatrix}^{T}$.

\section{Analysis of nonlinear behaviour}\label{sec:analysis}

Section \ref{sec:linear} focussed on the behaviour of the linear part of the Hasegawa--Wakatani equations, which is a necessary first step to understanding the nonlinear behaviour. The highly structured nature of the Hasegawa--Wakatani equations means we are able to systematically use centre manifold theory to predict the behaviour of the nonlinear system.

To apply the ideas behind centre manifold theory we firstly need to build an appropriate {\it suspended system}. 
 
\subsection{Suspended system}\label{sec:suspend}

The matrix in \eqref{eqn:lin1} depends on the parameters $\alpha$, $\kappa$ and $\beta_{\phi}$ and $\beta_{\rho}$. Let $\hat \epsilon$ represent one of the previously mentioned parameters and set $\epsilon = \hat \epsilon - \epsilon^\star$. We will use $\epsilon^\star$ in Section \ref{sec:eg} to reference a particular choice of parameters that gives eigenvalues with zero real parts. 
A suspended system corresponding to the system given in \eqref{eqn:lin1} is
\begin{eqnarray}\label{eqn:suspend}
\frac{\partial }{\partial t} \begin{bmatrix}\bm{\Phi} \\ \bm{R} \end{bmatrix}
&=&  L(\epsilon^\star) \begin{bmatrix}\bm{\Phi} \\ \bm{R}\end{bmatrix} + \epsilon G \begin{bmatrix}\bm{\Phi} \\ \bm{R}\end{bmatrix} + F(\bm{\Phi}, \bm{R}),\\
\frac{\partial{\epsilon}}{\partial t} & = & 0.\nonumber
\end{eqnarray}

The $G$ matrix represents the interaction of the parameter $\hat \epsilon$ with $\bm{R}$ and $\bm{\Phi}$.  For example, if $\hat \epsilon = \alpha$, we can rewrite \eqref{eqn:fhw1} and \eqref{eqn:fhw2} as
\begin{eqnarray*}
\frac{\partial}{\partial t} \Phi_{\bm{k}} & = &\frac{(N(\bm{\Phi}, D_N\bm{\Phi}))_{\bm{k}}}{\bm{k}^2_+}  - \frac{\alpha^\star\left(\Phi_{\bm{k}} - R_{\bm{k}} \right)}{\bm{k}^2_+} -  \beta_{\phi} \bm{k}^{2p} \Phi_{\bm{k}} \\
&& - \frac{(\alpha-\alpha^\star)\left(\Phi_{\bm{k}} - R_{\bm{k}} \right)}{\bm{k}^2_+}, \\
\frac{\partial}{\partial t} R_{\bm{k}} & = & - (N(\bm{\Phi}, \bm{R}))_{\bm{k}} + \alpha^{\star} \left(\Phi_{\bm{k}} - R_{\bm{k}}\right) - ik_y \kappa \Phi_{\bm{k}} - \beta_{\rho} \bm{k}^{2p}  R_{\bm{k}} \\
&& +(\alpha -\alpha^{\star})\left(\Phi_{\bm{k}} - R_{\bm{k}}\right), 
\end{eqnarray*}
and
\[
G  \begin{bmatrix}\bm{\Phi} \\ \bm{R} \end{bmatrix} =  \begin{bmatrix} -(\diag(\bm{k}^2_+))^{-1} & (\diag(\bm{k}^2_+))^{-1} \\ I & -I \end{bmatrix}\begin{bmatrix} \bm{\Phi} \\  \bm{R}\end{bmatrix}.
\]

 The eigenvalues  and eigenvectors of $L(\epsilon^\star)$ are discussed in Sections \ref{sec:eigenvalue} and \ref{sec:eigenvector}.
To help with the notation we will place the eigenvectors into the columns of the matrix $P$. In particular,
\[
P = \begin{bmatrix}-(A-\diag(\lambda^+_{\bm{k}}))^{-1}B &-(A-\diag(\lambda^-_{\bm{k}}))^{-1}B \\
I &I  \end{bmatrix}.
\]

Due to the simple structure of $P$ we can find
\[
P^{-1} = \begin{bmatrix}-B^{-1}(A-\diag(\lambda^-_{\bm{k}}))E(A-\diag(\lambda^+_{\bm{k}})) & -B^{-1}(A-\diag(\lambda^+_{\bm{k}}))EB \\
B^{-1}(A-\diag(\lambda^-_{\bm{k}}))E(A-\diag(\lambda^+_{\bm{k}})) & B^{-1}(A-\diag(\lambda^-_{\bm{k}}))EB \end{bmatrix},
\]
where
\[
E = \left(\diag(\lambda^{+}_{\bm{k}})-\diag(\lambda^{-}_{\bm{k}}) \right)^{-1},
\]
which is non-singular. The entries of $P^{-1}$ may look complicated, but they are all diagonal matrices.

Let \[\Lambda = \begin{bmatrix}\diag(\lambda_{\bm{k}}^{+}) & 0 \\ 0 & \diag(\lambda_{\bm{k}}^{-})\end{bmatrix}.\]
We then have the following eigendecomposition
\[
L(\epsilon^\star) = P \Lambda P^{-1}.
\]

Now let $\Pi$ be a permutation matrix such that the first $a$ entries along the diagonal of $\bar \Lambda = \Pi \Lambda \Pi^{-1}$ have zero real parts and the remaining $b = 2m^2-a$ entries have non-zero real components. Multiplying both sides of \eqref{eqn:suspend} by $Q = \Pi P^{-1}$ gives
\begin{eqnarray}
\frac{\partial }{\partial t} \bm{X}
&=& \Lambda_X \bm{X} + F_X(\bm{X}, \bm{Y})\label{eqn:centre1}\\
\frac{\partial }{\partial t} \bm{Y} &=& \Lambda_Y \bm{Y} + F_Y(\bm{X}, \bm{Y})\label{eqn:centre2},
\end{eqnarray}
where
\[
Q \begin{bmatrix}\bm{\Phi} \\ \bm{R} \end{bmatrix} = \begin{bmatrix}\bm{X} \\ \bm{Y}\end{bmatrix}, \quad \bm{X} \in \C^a, \bm{Y}\in \C^b,
\]
\[
\bar \Lambda = \begin{bmatrix}\Lambda_X & 0 \\ 0 & \Lambda_Y\end{bmatrix}, \quad \Lambda_X \in \C^{a,a}, \Lambda_Y \in \C^{b,b},
\]
\[
F_X(\bm{X}, \bm{Y}) = \begin{bmatrix}I_{a,a} & 0 \end{bmatrix}\left(\epsilon M \begin{bmatrix}\bm{X} \\ \bm{Y}\end{bmatrix}+\bar F(\bm{X}, \bm{Y})\right),\]
\[F_Y(\bm{X}, \bm{Y}) = \begin{bmatrix}0 & I_{b,b} \end{bmatrix}\left(\epsilon M \begin{bmatrix}\bm{X} \\ \bm{Y}\end{bmatrix}+\bar F(\bm{X}, \bm{Y})\right),
\]
\[
\bar F(\bm{X}, \bm{Y}) = QF(\bm{\Phi}, \bm{R}),
\]
and $M = Q G Q^{-1}$.
$I_{a,a}$ and $I_{b,b}$ are  $a \times a$ and $b \times b$ identity matrices respectively.

There is some freedom in the choice of $\Pi$, as long as the first $a$ entries of the diagonal of $\bar \Lambda$ have zero real components. We define $\Pi$ as
\[
\Pi = \begin{bmatrix}\Pi_+ & 0 \\
			0 & I_{m^2, m^2}\end{bmatrix}
\]
where $\Pi_+ \in \R^{m^2, m^2}$ is defined such that 
\[-\left|\Re(\Pi_+\diag(\lambda_{\bm{k}}^+)\Pi_+^{-1})_{j_1}\right| > -\left|\Re(\Pi_+\diag(\lambda_{\bm{k}}^+)\Pi_+^{-1})_{j_2}\right|\] if $j_1 < j_2$.

Equations \eqref{eqn:nonlin_1} and \eqref{eqn:nonlin_2} show that $F(\bm{0},\bm{0}) = \bm{0}$, which gives $\bar F\left(\bm{0}\right) = \bm{0}$. Furthermore $D F(\bm{0}, \bm{0}) = 0$ (where $DF$ is the Jacobian of $F$), so $D\bar F\left(\bm{0}\right) = \bm{0}$.
Therefore we can apply centre manifold theory as described in \cite{Carr:1981} to predict the behaviour of \eqref{eqn:lin1} by projecting a system defined on $\C^{2m^2, 2m^2}$ down onto a much lower dimensional space of size $\C^{a, a}$.

\subsection{Centre manifold theory}\label{sec:reduce}

Consider the system
\begin{equation}\label{eqn:reduce}
\frac{\partial }{\partial t} \bm{X}
= \Lambda_X \bm{X} + F_X(\bm{X}, H(\bm{X}, \epsilon))
\end{equation}
where $H : \C^{a+1} \rightarrow \C^{b}$ satisfies
\begin{eqnarray*}
DH(\bm{X}, \epsilon)\left[\Lambda_X \bm{X} + F_X(\bm{X}, H(\bm{X}))\right] &= & \Lambda_Y H(\bm{X}, \epsilon) + F_Y(\bm{X}, H(\bm{X}, \epsilon)),\\
H(\bm{0}, 0) & = & 0, \\
DH(\bm{0}, 0) & = & 0.
\end{eqnarray*}

$\Lambda_X$ is a diagonal matrix whose entries have zero real parts and $\Lambda_Y$ is a diagonal matrix whose entries have non-zero real components. We expect the value of $a$ to be relatively small. 

In the analysis of the suspended system shown below we project down onto the subspace spanned by the eigenvectors corresponding to the eigenvalues stored in $\Lambda_X$. This projection is carried out by the $Q = \Pi P^{-1}$ matrix. By considering the block structure of the $P^{-1}$ matrix we see that $X_{\bm{k}}$ is some linear combination of
$\Phi_{\bm{k}}$ and $R_{\bm{k}}$. In other words, analysis carried out on mode $\bm{k}$ in the reduced system will, directly, tell us about the behaviour of mode $\bm{k}$ in the full system (By the full system we mean the system defined in \eqref{eqn:fhw1} and \eqref{eqn:fhw2}). Or put more crudely, there is no mixing of the modes as we move between the two different spaces.

By Theorem 2, Section 2.4 of \cite{Carr:1981}, the solutions of \eqref{eqn:centre1} and \eqref{eqn:centre2} are stable (asymptotically stable) (unstable) if the solutions of system \eqref{eqn:reduce} are stable (asymptotically stable) (unstable).

Using the above conditions to find $H$ is, in general, very difficult, instead we find an approximation $\Psi : \C^{a+1} \rightarrow \C^{b}$. Define
\begin{eqnarray}
(N\Psi)(\bm{X}, \epsilon) &=& D\Psi(\bm{X}, \epsilon)\left[\Lambda_X \bm{X} + F_X(\bm{X}, \Psi(\bm{X}, \epsilon))\right]\nonumber\\
&& - \left(\Lambda_Y \Psi(\bm{X}, \epsilon) + F_Y(\bm{X}, \Psi(\bm{X}, \epsilon))\right).\label{eqn:N_psi}
\end{eqnarray}
Suppose $\Psi(\bm{0}, 0) = 0$, $D\Psi(\bm{0}, 0) = 0$ and that 
\[\|(N\Psi)(\bm{X}, \epsilon)\|_{b} = \mathcal{O}\left(\left\|\begin{bmatrix}\bm{X}  \\ \epsilon \end{bmatrix}\right\|_{a+1}^q\right)\] 
as $\|\begin{bmatrix}\bm{X}^T & \epsilon\end{bmatrix}^T \|_{a+1} \rightarrow 0$ where $q > 1$. From Theorem 3, Section 2.5 of \cite{Carr:1981}, 
\[
\|H(\bm{X}, \epsilon)-\Psi(\bm{X}, \epsilon)\|_b =  \mathcal{O}\left(\left\|\begin{bmatrix}\bm{X} \\ \epsilon \end{bmatrix}\right\|_{a+1}^q\right).
\]
$\|.\|_{a+1}$ and $\|.\|_b$ are some norms defined on $\C^{a+1}$ and $\C^b$ respectively.

For now, lets assume $\Psi(\bm{X}, \epsilon) = \epsilon \Psi_L \bm{X} + \Psi_N(\bm{X})$ where
$\Psi_L \in \C^{b,a}$ is a matrix that will be determined below, and $\Psi_N:\C^b \rightarrow \C^a$ is a quadratic of the form $\Psi_N(\bm{X}) = \begin{bmatrix} 0 & I_{bb}\end{bmatrix} \Upsilon_N(\bm{X})$ where $(\Upsilon_N(\bm{X}))_{\bm{k}} = \sum_{\bm{j}_1}\sum_{\bm{j}_2}\xi_{(\bm{j}_1, \bm{j}_2)}^{\bm{k}} X_{\bm{j}_1}X_{\bm{j}_2}$. Rewrite $M = QGQ^{-1}$ as
\[
M = \begin{bmatrix} M_{11} & M_{12}\\
M_{21} & M_{22}
\end{bmatrix}
\]
where $M_{11} \in C^{a, a}$.

To ensure a good approximation to the centre manifold $H$ we want to make $(N\Psi)(\bm{X}, \epsilon)$ in Equation \eqref{eqn:N_psi} \lq small\rq . After substituting  $\Psi(\bm{X}, \epsilon)$ into $(N\Psi)(\bm{X}, \epsilon)$, we firstly collect all of the terms that are linear in $\bm{X}$. These terms are
\[
\epsilon \Psi_L \Lambda_X \bm{X}+ \epsilon^2\Psi_L M_{11}\bm{X} + \epsilon^3\Psi_L M_{12}\Psi_L\bm{X} - \left(\epsilon \Lambda_Y \Psi_L \bm{X} + \epsilon M_{21} \bm{X} + \epsilon^2 M_{22} \Psi_L \bm{X}\right). 
\]
As $\Lambda_X$, $\Lambda_Y$, $M_{11}$ and $M_{22}$ are all diagonal matrices it is straight forward to find $\Psi_L$ so that it solves the equation
\[\epsilon \Psi_L \Lambda_X \bm{X}+ \epsilon^2\Psi_L M_{11}\bm{X}  - \left(\epsilon \Lambda_Y \Psi_L \bm{X} + \epsilon M_{21} \bm{X} + \epsilon^2 M_{22} \Psi_L \bm{X}\right) = 0.\] This would leave us with terms that are $\mathcal{O}(\epsilon^3\bm{X})$. However, to help simplify the analysis of the reduced system we instead choose $\Psi_L$ so that it solves $\epsilon \Psi_L \Lambda_X \bm{X}  - \left(\epsilon \Lambda_Y \Psi_L \bm{X} + \epsilon M_{21} \bm{X} \right) = 0.$ In particular, we set
\[
(\Psi_L)_{ij} = \frac{(M_{21})_{ij}}{(\Lambda_X)_{jj}-(\Lambda_Y)_{ii}}.
\]
The remaining terms are $\mathcal{O}(\epsilon^2\bm{X})$.

Referring back to $(N\Psi)(\bm{X}, \epsilon)$, the terms that are second order in $\bm{X}$ are
\begin{eqnarray*}
\epsilon^2\Psi_LM_{12}\Psi_N(\bm{X})&+&\epsilon\Psi_L\begin{bmatrix}I_{a,a} &0 \end{bmatrix}\bar F(\bm{X}, \epsilon \Psi_L\bm{X})\\
&+& D\Psi_N(\bm{X})\Lambda_{X}\bm{X} + \epsilon D\Psi_N(\bm{X})M_{11}\bm{X} \\
&+& \epsilon^2 D\Psi_N(\bm{X})M_{12}\Psi_L\bm{X}\\
&-&\left(\Lambda_Y \Psi_N(\bm{X})+\epsilon M_{22}\Psi_N(\bm{X})+\begin{bmatrix}0 & I_{b,b} \end{bmatrix}\bar F(\bm{X}, \epsilon \Psi_L\bm{X})\right).
\end{eqnarray*}

Setting $\Upsilon_N(\bm{X})$ so that
\begin{equation}\label{eqn:secondorder}
 D\Upsilon_N(\bm{X})\left(\Lambda_{X} + \epsilon M_{11}\right)\bm{X}
-\begin{bmatrix} 0 & 0 \\ 0 & \left(\Lambda_Y + \epsilon M_{22}\right)\end{bmatrix} \Upsilon_N(\bm{X}) = \bar F(\bm{X}, \epsilon \Psi_L\bm{X}),
\end{equation}
leaves remainder terms that are $\mathcal{O}(\epsilon^2\bm{X}\bm{X}^T)$. Now
\[
\left(D\Upsilon_N(\bm{X})\right)_{\bm{k}, \bm{j}} = \sum_{\bm{j}_2}\left[\xi_{(\bm{j},\bm{j}_2)}^{\bm{k}}+\xi_{(\bm{j}_2,\bm{j})}^{\bm{k}}\right]X_{\bm{j}_2}.
\]
Consider a mode $\bm{k}$ such that $(\Lambda_Y)_{\bm{k}}$ is defined.
The entry of the left hand side of Equation \eqref{eqn:secondorder} corresponding to mode $\bm{k}$ can be written as
\begin{equation}\label{eqn:rhs_quad} \sum_{\bm{j}_1}\sum_{\bm{j}_2}\xi_{(\bm{j}_1,\bm{j}_2)}^{\bm{k}} c_{(\bm{j}_1, \bm{j}_2)}^{\bm{k}} X_{\bm{j}_1}X_{\bm{j}_2},
\end{equation}
where $c_{(\bm{j}_1, \bm{j}_2)}^{\bm{k}} = \left(\Lambda_{X} + \epsilon M_{11}\right)_{(\bm{j}_1, \bm{j}_1)}+ \left(\Lambda_{X} + \epsilon M_{11}\right)_{(\bm{j}_2, \bm{j}_2)}- \left(\Lambda_Y + \epsilon M_{22}\right)_{(\bm{k},\bm{k})}$. Using 
\[
\begin{bmatrix}\bm{\Phi} \\ \bm{R}\end{bmatrix} = Q^{-1} \begin{bmatrix} I_{a, a}\\ \epsilon \Psi_L \end{bmatrix} \bm{X}
\]
with Equations \eqref{eqn:nonlin_1} and \eqref{eqn:nonlin_2} we see that
\begin{equation}\label{eqn:lhs_quad}
\left(\bar F(\bm{X}, \epsilon \Psi_L\bm{X})\right)_{\bm{k}} = \sum_{\bm{j}_1}\sum_{\bm{j}_2} f_{(\bm{j}_1, \bm{j}_2)}^{\bm{k}} X_{\bm{j}_1}X_{\bm{j}_2},
\end{equation}
where the constants $f_{(\bm{j}_1, \bm{j}_2)}^{\bm{k}}$ are known (can be calculated). Comparing Equations \eqref{eqn:rhs_quad} and \eqref{eqn:lhs_quad} suggests
\[
\xi_{(\bm{j}_1,\bm{j}_2)}^{\bm{k}} =  f_{(\bm{j}_1, \bm{j}_2)}^{\bm{k}}/c_{(\bm{j}_1, \bm{j}_2)}^{\bm{k}}.
\]

So $\Psi(\bm{X}, \epsilon) = \epsilon \Psi_L\bm{X} + \Psi_N(\bm{X})$ gives an third order approximation to $H$.

Substituting $\Psi$ back into \eqref{eqn:reduce}, the reduced system becomes
\begin{eqnarray}
\frac{d}{dt}\bm{X} &=& \Lambda_X \bm{X} + \epsilon\left(M_{11}\bm{X}+M_{12}\Psi(\bm{X}, \epsilon)\right)+\begin{bmatrix}I_{a,a} & 0 \end{bmatrix}  \bar F\left(\bm{X},\Psi(\bm{X}, \epsilon)\right)\nonumber\\
&=&  \left[\Lambda_X \bm{X} + \epsilon M_{11}\bm{X}+\epsilon^2 M_{12}\Psi_L\bm{X}\right]\nonumber\\&&\quad
+ \epsilon M_{12}\Psi_N(\bm{X}, \epsilon)+\begin{bmatrix}I_{a,a} & 0 \end{bmatrix}  \bar F\left(\bm{X},\Psi(\bm{X}, \epsilon)\right)\label{eqn:low_order}.
\end{eqnarray}
We can study the behaviour of the above $a \times a$ system to infer the behaviour of the original system in \eqref{eqn:centre1} and \eqref{eqn:centre2}, and consequently \eqref{eqn:fhw1} and \eqref{eqn:fhw2}.

\section{Application of centre manifold analysis - periodic boundaries}\label{sec:eg}

To help better understand the behaviour of the Hasegawa--Wakatani equations, the nonlinear analysis developed in Section \ref{sec:reduce} is firstly applied to example problems with periodic boundary conditions. The zero boundary case is studied in Section \ref{sec:eg_zero}; the results differ greatly from those presented in this section.

In Section \ref{sec:eg_beta} the modes are analysed individually. That is, in the centre manifold analysis we only concentrate on those modes whose eigenvalues have zero real components for a given values of $\alpha$. In Section \ref{sec:period_model} we present some experiments that are designed to take the interaction between the modes into account.

 We will ignore the $\bm{k} = (0,0)$ mode in the following discussion as it does not influence the behaviour of either the full or reduced systems. To simplify the discussion we set $\beta_{\rho}$ = $\beta_{\phi}$ = $\beta$.

\subsection{Individual mode study}\label{sec:eg_beta}

In the first example we set $n$ = 16, $\kappa$ = 1.5, $p$ = 1 and try varying $\alpha$. The $G$ matrix is as defined previously in Section \ref{sec:suspend}.

A small script was written in Scilab\footnote{\url{http://www.scilab.org/}} to evaluate the eigenvalues. The plots in Figure \ref{fig:eig_beta1} show the maximum value of the real part of the non-zero eigenvalues of $L$ for $\beta$ = 0.1, 0.05, 0.01, 0.005 and 0.001, and varying $\alpha$.  In the centre manifold theory we are interested in the choices of $\alpha$ that give eigenvalues with zero real components. It may not be so obvious in the plots, but the graphs cross the $x$-axis twice, once very close to the $y$-axis. That is, for each $\beta$ there is an $\alpha_u$ where the $\Re\left(\lambda^{\pm}_{\bm{k}}\right)< 0$ for all $\bm{k}$ if $\alpha > \alpha_u$ and there is an  $\alpha_l$ where the $\Re\left(\lambda^{\pm}_{\bm{k}}\right)< 0$ for all $\bm{k}$ if $\alpha < \alpha_l$. In other words, the solutions converges to zero  if $\alpha > \alpha_u$ and $\alpha > \alpha_l$. Consequently, we are interested in the region $\alpha_l < \alpha < \alpha_u$. We concentrate on the results for larger values of $\alpha$ and now set $\beta = 0.001$.

In this case $L$ is a matrix of size $2046 \times 2046$. If we were unable to exploit the sparse structure of the matrix and had to calculate the eigenvalues and eigenvectors numerically the computational cost would be exorbitant.
 
\begin{figure}
\begin{center}
\includegraphics[scale = 0.35, angle = 270]{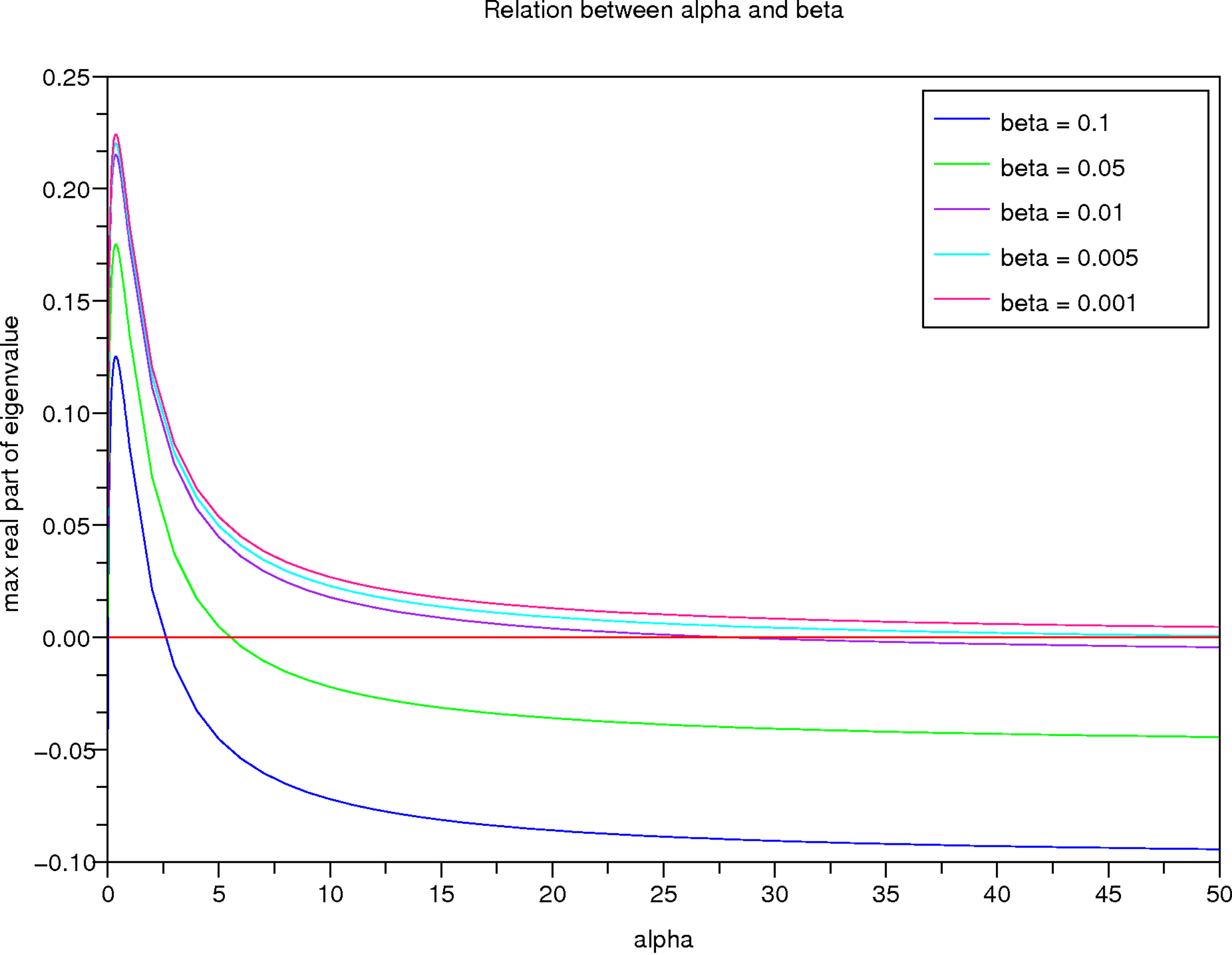}
\end{center}
\caption{Maximum real part of the non-zero eigenvalues of $L$ as $\alpha$ is increased with different values of $\beta$}\label{fig:eig_beta1}
\end{figure}

The numerical simulation of the original full system in \eqref{eqn:fhw1} and \eqref{eqn:fhw2} were carried out using the Fortran 90 code on a grid of size $64 \times 64$.

For the initial condition $\bm{\Phi}^0$, the coefficients of $\Phi^0_{\bm{k}}$ were chosen randomly if $0 \le |k_x|\le 5$ and $1 \le |k_y| \le 5$, otherwise $\Phi^0_{\bm{k}}$ is zero. Furthermore $|\Phi^0_{\bm{k}}| \le 0.01$, $\bm{R}^0 = \bm{\Phi}^0$ and $\Phi^0_{(k_x, k_y)}$ = $\overline{\Phi}^0_{(-k_x, -k_y)}$ (for all $\bm{k}$).


Starting with a large value of $\alpha$ and slowly reducing it, we found the first eigenvalues to have zero real components are $\lambda^{+}_{(0, \pm 1)}$ when $\alpha = 281.2475$. (So, $\alpha_u$ =281.2475.)  According to  the centre manifold analysis the reduced system is the following linear system (we are ignoring the $\bm{k} = (0,0)$ mode)
\begin{eqnarray}
\frac{d}{dt}X_{(0, -1)} &=&  \left(7.500\times 10^{-1}i + \epsilon  (-3.556\times 10^{-6} +  1.896\times10^{-8}i)\right)X_{(0,-1)}, \nonumber\\
&& +\epsilon^2 \left(6.321\times10^{-9} - 5.899\times10^{-11}i \right) X_{(0,-1)} ,\label{eqn:beta_01}\\
\frac{d}{dt}X_{(0, 1)} &=&\frac{d}{dt}\overline{X}_{(0, -1)}. \nonumber
\end{eqnarray}

The first point to note is that the nonlinear terms in Equation \eqref{eqn:low_order} are zero for this particular example. To see why, firstly observe that the vector $\begin{bmatrix} \bm{\Phi}_X & \bm{R}_X\end{bmatrix}^T = Q^{-1}\begin{bmatrix} \bm{X} & \Psi(\bm{X}, \epsilon)\end{bmatrix}^T$ will be zero everywhere except for those entries corresponding to the $\bm{k} = (0, \pm 1)$ modes. Therefore the inverse Fourier transform of $ \bm{\Phi}_X$ and $\bm{R}_X$ are one dimensional, real, functions of $y$. The Poisson bracket applied to such functions is zero. 

If $\alpha \gtrapprox 281.2475$ ($\epsilon \gtrapprox 0$), the real part of the eigenvalues of Equation \eqref{eqn:beta_01} are negative, indicating that the solution should converge towards zero. The numerical solution of both the reduced and full systems did converge towards zero.

Given the small values of the real components of the eigenvalues of \eqref{eqn:beta_01}, the solutions of the full system should be stable when $\alpha \lessapprox 281.2475$. Figure \ref{fig:beta280} shows the real part of  ${\Phi}_{(0,1)}$ and ${R}_{(0,1)}$ when $\alpha = 280$. 
Note however, as $\alpha$ is decreased further the real components of the eigenvalues will increase, and possibly become positive, so for certain values of $\alpha$ we may see exponential growth in the $\bm{k} = (0, \pm 1)$ modes.

\begin{figure}
\begin{center}
\input{figure/beta_mode_280.tex}
\end{center}
\caption{Real part of  $\bm{\Phi}_{(0,1)}$ and $\bm{R}_{(0,1)}$ for $\alpha$ = 280  and $0 \le t \le 5000$.}\label{fig:beta280}
\end{figure}



Continuing to reduce $\alpha$ we notice $\Re\left(\lambda^{+}_{(\pm 1, \pm 1)}\right) = 0$ at $\alpha = 83.326665$. 
The reduced system is
\begin{eqnarray}
\frac{d}{dt}X_{(-1, -1)} &=& \left(5.000\times 10^{-1}i  + \epsilon  ( -2.400\times 10^{-5} + 3.840\times 10^{-7}i)\right)X_{(-1,-1)}\nonumber\\
&& +\epsilon^2 \left(1.919\times 10^{-7} - 5.375\times 10^{-9}i\right) X_{(-1,-1)} \nonumber\\
&& + \epsilon \left(M_{12}\Psi_N(\bm{X}, \epsilon)\right)_{(-1, -1)}+ \bar F(\bm{X}, \Psi(\bm{X}, \epsilon))_{(-1, -1)},\label{eqn:beta_11}\\
\frac{d}{dt}X_{(-1, 1)} &=& \frac{d}{dt} \overline{X}_{(-1, -1)},\nonumber\\
\frac{d}{dt}X_{(1, -1)} &=& \frac{d}{dt} {X}_{(-1, -1)},\nonumber\\
\frac{d}{dt}X_{(1, 1)} &=& \frac{d}{dt} \overline{X}_{(-1, -1)}.\nonumber
\end{eqnarray}

 To solve Equation \eqref{eqn:beta_11} we used the Scilab ordinary differential equation solver.

For $\epsilon > 0$ the eigenvalues of the linear part of the reduced system have negative real components, so the solution should converge towards zero. The numerical solutions of the reduced system converged to zero for such values of $\epsilon$. 


\begin{figure}
\begin{center}
\includegraphics[height = 12cm, width = 3.5cm, angle = 270]{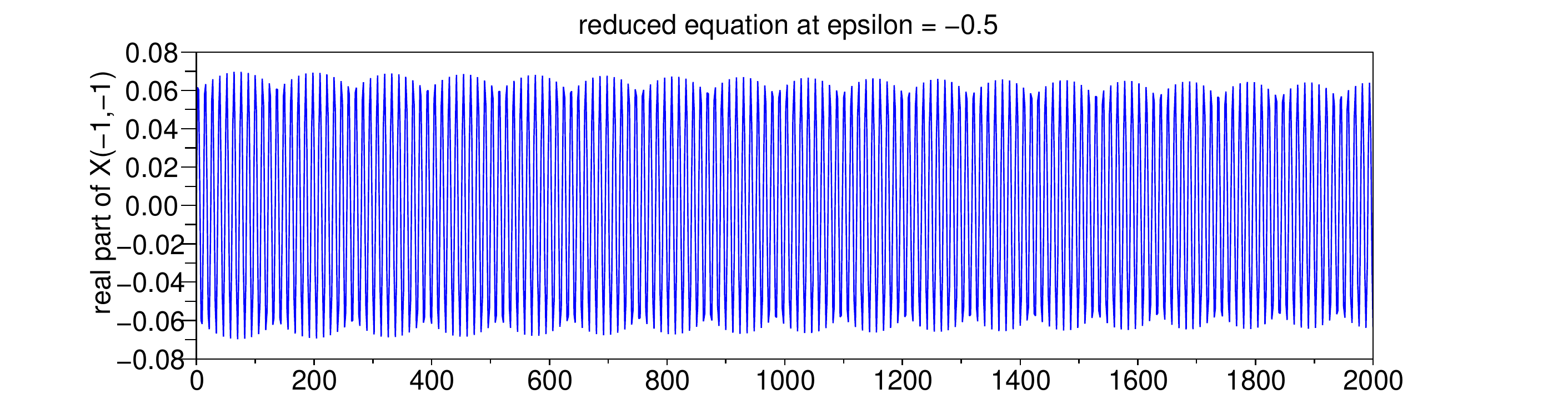}
\end{center}
\caption{Real part of $X_{(\pm 1, \pm 1)}$ when $\epsilon = -0.5$. }\label{fig:reduce1_m}
\end{figure}

 The real part of $X_{(1,1)}$ is shown in 
Figure \ref{fig:reduce1_m} for $\epsilon = -0.5$
The results imply the $\bm{k} = (\pm 1, \pm 1)$ modes are stable when $\alpha \approx  83.326665$, which was the case in the numerical simulations of the full system.
  

The $\bm{k} = (\pm 1, \pm 1)$ mode may be stable, but according to Equation \eqref{eqn:beta_01} the $\bm{k} = (0, \pm 1)$ mode will grow when $\alpha \approx 83.326665$, as shown in Figure \ref{fig:beta83}. 

\begin{figure}
\begin{center}
\input{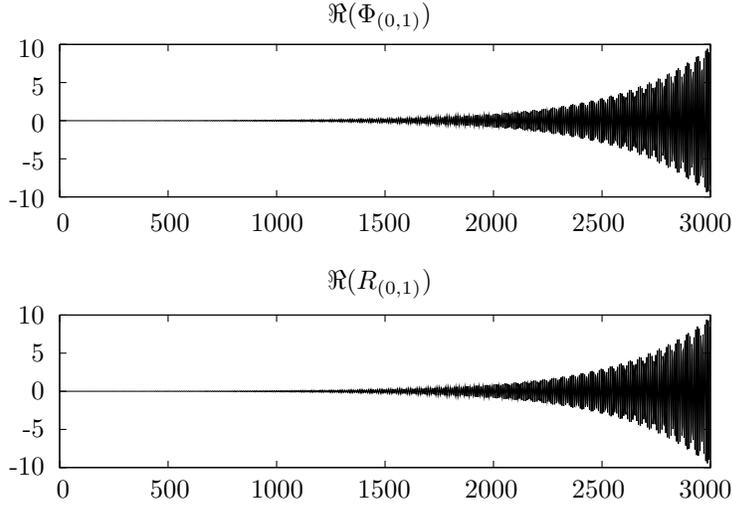}
\end{center}
\caption{Real part of  ${\Phi}_{(0,1)}$, and ${R}_{(0,1)}$ for $\alpha$ = 83  and $0 \le t \le 3000$.}\label{fig:beta83}
\end{figure}

Reducing the value of $\alpha$ further the next eigenvalue to have real zero components is $\lambda^{+}_{(0, \pm 2)}$ at $\alpha =  71.984$. As with the $\bm{k} = (0, \pm 1)$ case, the nonlinear term drops out and the reduced system is
\begin{eqnarray}
\frac{d}{dt}X_{(0, -2)} &=&  \left( 5.999\times 10^{-1}i + \epsilon  ( -5.554\times10^{-5} + 1.482\times 10^{-6}i)\right)X_{(0,-2)} \nonumber\\
&& +\epsilon^2 \left( 6.165\times 10^{-7} - 2.878\times 10^{-8}i\right) X_{(0,-2)}, \\\label{eqn:beta_02}
\frac{d}{dt}X_{(0, 2)} &= & \frac{d}{dt} \overline{X}_{(0, -2)}.\nonumber
\end{eqnarray}

 
 
 
 
 
 
 
 
 
 
 

For negative values of $\epsilon$ ($\alpha <  71.984$) the eigenvalues of \eqref{eqn:beta_02} have positive real components so we expect to see a growth in the $\bm{k} = (0, \pm 2)$ modes.

The next eigenvalue to obtain zero real components is $\lambda^+_{(\pm 1, \pm 2)}$ at $\alpha =  41.64583$. We used the Scilab's ordinary differential equation solver to find the solutions of the reduced system when $\epsilon < 0$.   The results in Figure \ref{fig:reduce2_m} indicate this is a stable mode when $\alpha \approx 41.64583$. Once again, exponential growth in the $\bm{k} = (0, \pm 1)$ mode is evident when $\alpha = 41$.
 

\begin{figure}
\begin{center}
\includegraphics[height = 12cm, width = 3.5cm, angle = 270]{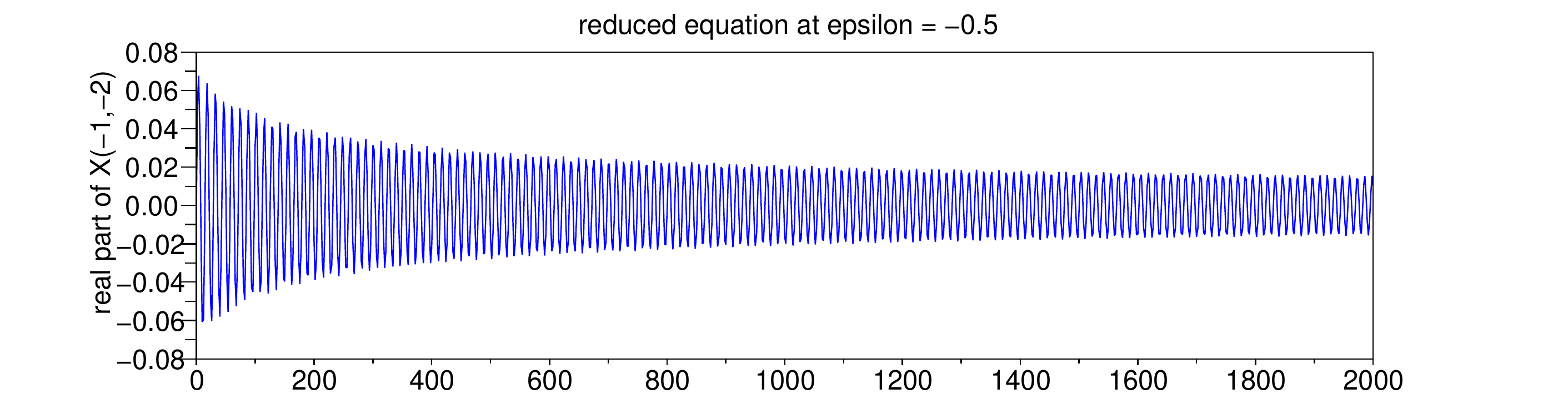}
\end{center}
\caption{Real part of $X_{(\pm 1, \pm 2)}$ when $\epsilon = -0.5$. }\label{fig:reduce2_m}
\end{figure}



It is important to carry out the numerical simulations on appropriate sized grids. For example, if we rerun the previous experiment on a grid of size $32 \times 32$, instead of $64 \times 64$, the $\bm{k} = (0, \pm 1)$ mode appears to be stable, as shown in  Figure \ref{fig:beta41_small}. This is incorrect, the solution should grow exponentially.
If we did not have the analytical results and just looked at the results for the $\bm{k} = (0, \pm 1)$ mode shown in Figure \ref{fig:beta41_small}, it would not be obvious that something is wrong. But, the apparent high frequency (noisy) solution  of $\Phi_{(\pm 1, \pm 1)}$ and $R_{(\pm 1, \pm 1)}$ in Figure \ref{fig:beta41_11_small} is an indication of the presence of numerical errors.

\begin{figure}
\begin{center}
\input{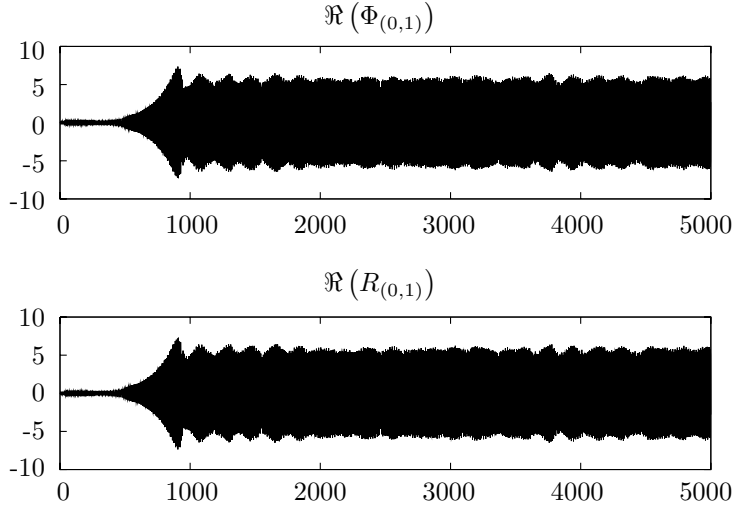}
\end{center}
\caption{Real part of  $\Phi_{(0,1)}$, and $R_{(0,1)}$ for $\alpha$ = 41, $0 \le t \le 5000$ and numerical grid of size $32 \times 32$.}\label{fig:beta41_small}
\end{figure}

\begin{figure}
\begin{center}
\input{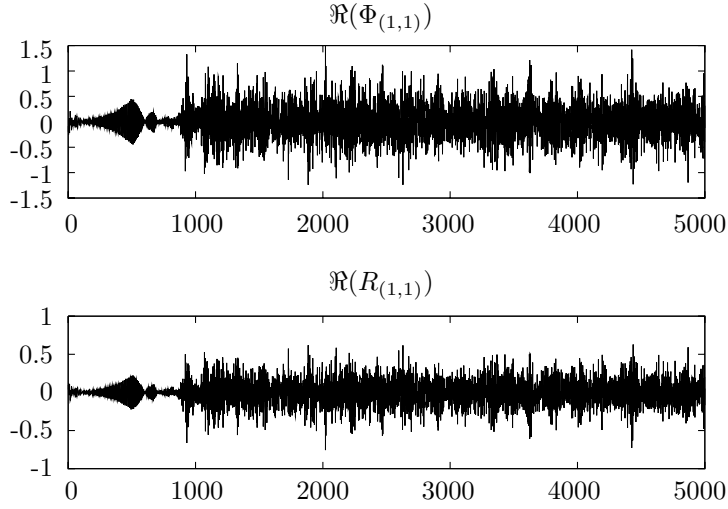}
\end{center}
\caption{Real part of  $\Phi_{(1,1)}$, and $R_{(1,1)}$ for $\alpha$ = 41, $0 \le t \le 5000$ and numerical grid of size $32 \times 32$.}\label{fig:beta41_11_small}
\end{figure}

The above results show how important it is to use a computational grid of the correct size. If the grid is too small the qualitative type of information extracted from the solutions can be very misleading.


Reducing $\alpha$ even further, the next eigenvalues to have zero real components are  $\lambda^{+}_{(0, \pm 3)}$ when  $\alpha = 20.20947$, $(\pm 1, \pm 3)$ when $\alpha = 15.168627$, $(\pm 2, \pm 2)$  at $\alpha = 12.310091$, $(\pm 2, \pm 1)$ at $\alpha = 10.39582$ etc. In other words, as $\alpha$ is reduced more modes influence the long term behaviour of the system. The size of the computational grid must be big enough to ensure these modes are accurately represented in the numerical computations.

\subsection{Solution of a model problem}\label{sec:period_model}

The experiments in Section \ref{sec:eg_beta} focussed on one particular mode at a time.  Here we concentrate on a choice of parameters where a large number of modes are expected to contribute to the long term behaviour of the system. In particular, we set $p$ = 2, $\alpha$ = 1 and $\kappa$ = 1. The experiments were run using different values of $\beta$ ($\beta = 10^{-2}$, $\beta = 10^{-3}$ and $\beta = 10^{-5}$). We also tried computational grids of size $64 \times 64$ to $256 \times 256$ to ensure that results are consistent. 

We know, analytically, that  if the initial conditions $\bm{\Phi}^*$ and $\bm{R}^*$ are chosen so that $\Phi^*_{\bm{k}}$ and $\bm{R}^*_{\bm{k}}$ are zero everywhere except the $\bm{k} = (0, \pm 1)$ mode, the solution will grow exponentially. This is also evident in the numerical simulations. Is this behaviour restricted to very specific choices of initial conditions? Or is it an indication of what we should expect in general?

Lets consider another set of initial conditions $\bm{\Phi}^0$ and $\bm{R}^0$ where the coefficients of $\Phi^0_{\bm{k}}$ are chosen randomly if $0 \le |k_x|\le 5$ and $2 \le |k_y| \le 5$; and $\Phi^0_{(0, \pm 1)}=0.1$. Otherwise $\Phi^0_{\bm{k}}$ is zero. Furthermore, the values are scaled if $\bm{k} \ne (0, \pm 1)$ so that $|\Phi^0_{\bm{k}}| \le \gamma$ . For small values of $\gamma$ the initial condition is close to $\bm{\Phi}^*$ and $\bm{R}^*$. Finally $\bm{R}^0 = \bm{\Phi}^0$ and $\Phi^0_{(k_x, k_y)}$ = $\overline{\Phi}^0_{(-k_x, -k_y)}$. 

We found that for $\gamma = 10^{-6}$ and $\gamma = 10^{-4}$ the solution showed exponential growth for all values of $\beta$. In other words, any initial conditions close to $\bm{\Phi}^*$ and $\bm{R}^*$ appear to be unstable.

When $\gamma = 10^{-2}$ the solution once again showed exponential growth for $\beta = 10^{-2}$, but it behaved differently when $\beta \le 10^{-3}$. As shown in Figure \ref{fig:periodic}, the $\bm{k} = (0, \pm 1)$ modes initially appeared to be stable, until about $t = 300$ when they once again grow exponentially. The growth appears to be primarily in the  $\bm{k} = (0, \pm 1)$ modes as no growth is was evident in the  $\bm{k} = (\pm 1, \pm 1)$ modes. 

\begin{figure}
\begin{center}
\input{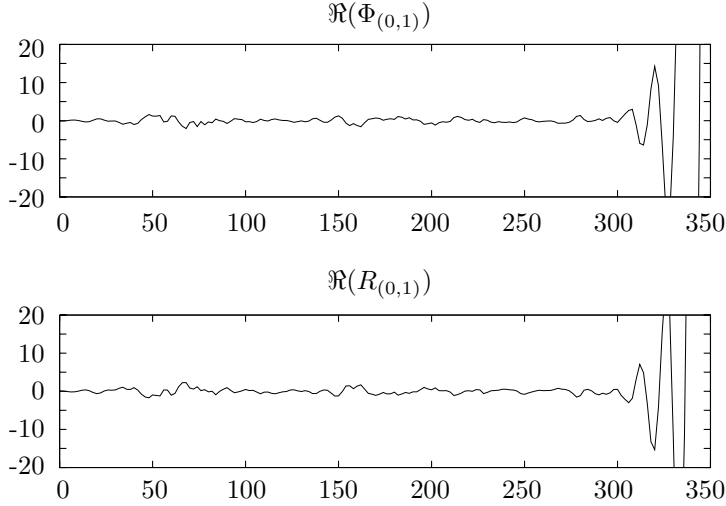}
\end{center}
\caption{Real part of  ${\Phi}_{(0, 1)}$, and ${R}_{(0, 1)}$ for $\gamma = 10^{-2}$  and $0 \le t \le 350$.}\label{fig:periodic}
\end{figure}



We also tried other initial conditions and obtained similar results to those mentioned above. We thus conclude that for periodic boundary conditions, the nodes sitting on the $(0, \pm k_y)$ boundary will always grow for certain choices of $\alpha$.

\section{Application of centre manifold analysis  - zero boundary conditions}\label{sec:eg_zero}

In Section \ref{sec:eg} we considered the case of periodic boundary conditions, in this section we assume that the $y$ boundaries are zero and the $x$ boundaries are periodic. Zero boundary conditions can be enforced by the use of $\sin$ transformations, which may in-turn be implemented through Fourier transformations, see \cite{Boyd:2001}. To implement a $\sin$ transformation of size $N$ we need to use a Fourier grid of size $2N$. 

When zero boundary conditions are applied along the $y$-axis the $\bm{k} = (0, k_y)$ modes will always be zero. We saw in Section \ref{sec:eg} that these modes played a special role.  We expect to see a marked changed in the behaviour of the system when  zero boundary conditions are enforced.

To find the initial conditions we firstly evaluate \[\bm{p} =  \sum_{k=1}^4\sum_{a,b=0}^{m-1} \gamma_k\eta_k \sin(2\pi k a/m)\sin(2\pi k b/m)+\sum_{a=0}^{m-1} \cos(2\pi k a/m)\] where $\gamma_k$ and $\eta_k$ are chosen randomly. We then set $\bm{\Phi} = s {\cal F}^{-1} \bm{p}/\| {\cal F}^{-1} \bm{p}\|_{1}$ where $\cal F$ is the Fourier transformation, $\|.\|_1$ is the L1 norm and $s$ is a scaling term.  Furthermore $\bm{R}$ = $\bm{\Phi}$ and $\Phi_{(k_x, k_y)}$ = $\overline{\Phi}_{(-k_x, -k_y)}$.

\subsection{Individual mode study}\label{sec:eg_beta_zero}

In this section we use the same parameters as in Section \ref{sec:eg_beta}, but apply zero boundary conditions along the $y$-axis. The numerical calculations were carried out on a grid of size $128 \times 128$. The scaling parameter $s$ is 0.01.

When $\alpha > 83.326665$ the real part of all of the eigenvalues of $L$ are negative, so the solution should converge to zero. The numerical simulations converged to zero.



We know that $\Re\left(\lambda^{+}_{(\pm 1. \pm 1)}\right) = 0$ at $\alpha =  83.326665$.
The solution of the reduced system shown in 
Figure \ref{fig:reduce1_m} suggests the $\bm{k} = (\pm 1, \pm 1)$ modes are stable when $\alpha \approx  83.326665$. The modes were stable in the numerical simulations.





However, we see from Equation \eqref{eqn:beta_11} that as $\alpha$ is decreased the real part of the eigenvalues of the linear system increase, and as shown in Figure \ref{fig:zero70_11} this growth may dominate the system. Such growth in the $\bm{k} = (\pm 1, \pm 1)$ modes was also confirmed by numerical solutions of the reduced system.

\begin{figure}
\begin{center}
\input{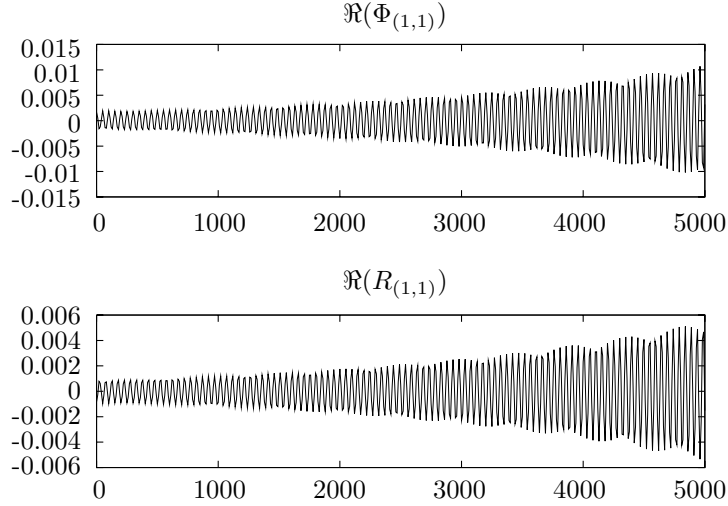}
\end{center}
\caption{Real part of  ${\Phi}_{(1, 1)}$ and ${R}_{(1, 1)}$ for $\alpha$ = 70  and $0 \le t \le 5000$.}\label{fig:zero70_11}
\end{figure}  

 Although the results for $\alpha = 70$ suggests the $\bm{k} = (\pm 1, \pm 1)$ will always grow exponentially, the plot for $\alpha = 41$ shown in Figure \ref{fig:zero41_11} contradicts such an assumption. We do not believe that the change in behaviour at about $t = 3000$ is a result of numerical error, but rather an example of the stable bifurcation described in Section \ref{sec:zero_model}. It was necessary to increase the size of computational domain to $512 \times 512$ to accurately capture the results as many of the higher order modes grew in the region where the system suddenly changed its behaviour. For relevance to the discussion in Section \ref{sec:zero_model}, we have also included a plot of the $\bm{k} = (\pm 1, 0)$ modes in Figure \ref{fig:zero41_10}.

\begin{figure}
\begin{center}
\input{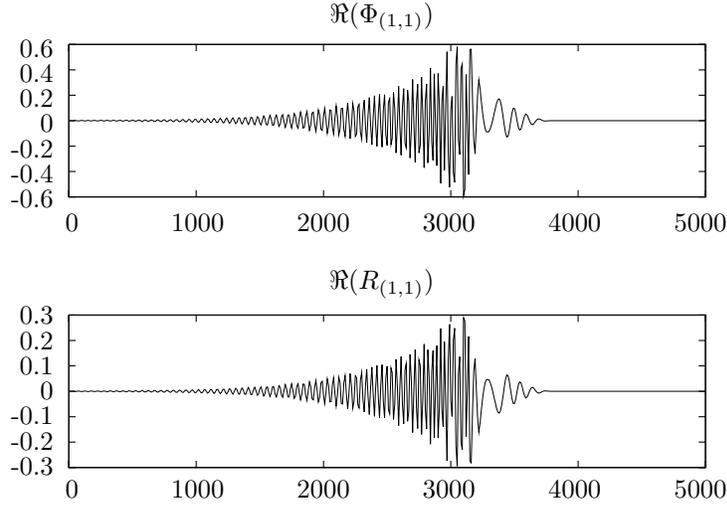}
\end{center}
\caption{Real part of  ${\Phi}_{(1, 1)}$ and ${R}_{(1, 1)}$ for $\alpha$ = 41  and $0 \le t \le 5000$.}\label{fig:zero41_11}
\end{figure}  

\begin{figure}
\begin{center}
\input{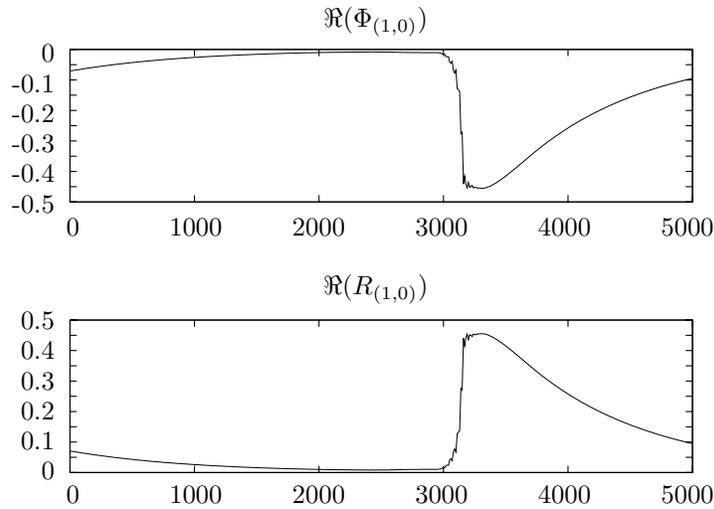}
\end{center}
\caption{Real part of  ${\Phi}_{(1, 0)}$ and ${R}_{(1, 0)}$ for $\alpha$ = 41  and $0 \le t \le 5000$.}\label{fig:zero41_10}
\end{figure}

\subsection{Solution of a model problem}\label{sec:zero_model}

 The model problem under consideration here is the same as the one in Section \ref{sec:period_model}, but with the $y$-axis set to zero. Consequently, the $(0, \pm 1)$ modes that showed exponential growth in Section \ref{sec:period_model} have been removed from the system. The initial conditions are the same as those mentioned in the start of Section \ref{sec:eg_zero}.


The properties of the solution varied greatly with $\beta$. Figure \ref{fig:zero1D2_11} shows the real part of  ${\Phi}_{(1, 1)}$ and ${R}_{(1, 1)}$ for $\beta = 0.01$. According to the solution of the corresponding reduced system, the $\bm{k} = (\pm 1, \pm 1)$ mode should be stable, which is in agreement with the numerical results shown in Figure \ref{fig:zero1D2_11}.

\begin{figure}
\begin{center}
\input{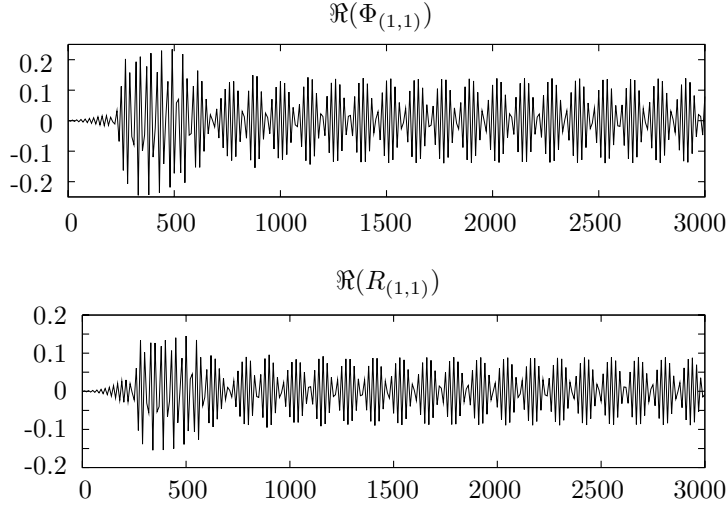}
\end{center}
\caption{Real part of  ${\Phi}_{(1, 1)}$ and ${R}_{(1, 1)}$ for $\beta = 0.01$  and $0 \le t \le 3000$.}\label{fig:zero1D2_11}
\end{figure}  




 Figure \ref{fig:zero1D3_11} shows the real part of  ${\Phi}_{(1, 1)}$ and ${R}_{(1, 1)}$ for $\beta = 0.001$. This time the solution of the reduced system shows a growth in the $\bm{k} = (\pm 1, \pm 1)$ modes.  We see from Figures \ref{fig:zero1D3_11} and \ref{fig:zero1D3_10} that the $(\pm 1, \pm 1)$ and $(\pm 1, 0)$ modes interact in such as way as to result in stable bifurcation.

\begin{figure}
\begin{center}
\input{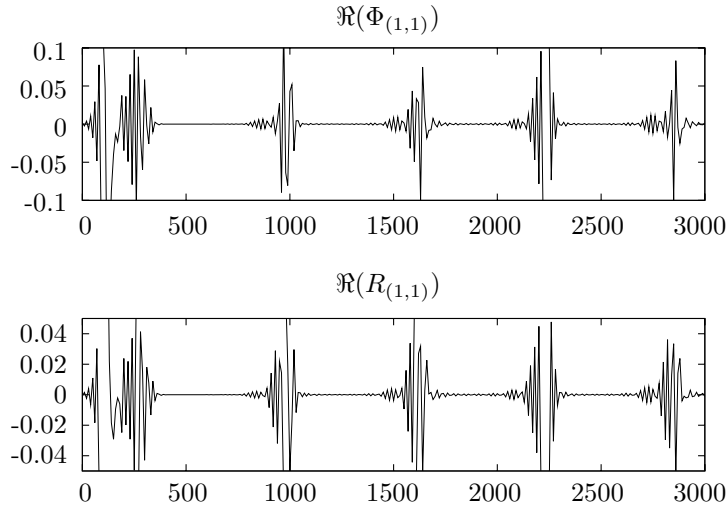}
\end{center}
\caption{Real part of  ${\Phi}_{(1, 1)}$ and ${R}_{(1, 1)}$ for $\beta = 0.001$  and $0 \le t \le 3000$.}\label{fig:zero1D3_11}
\end{figure} 

\begin{figure}
\begin{center}
\input{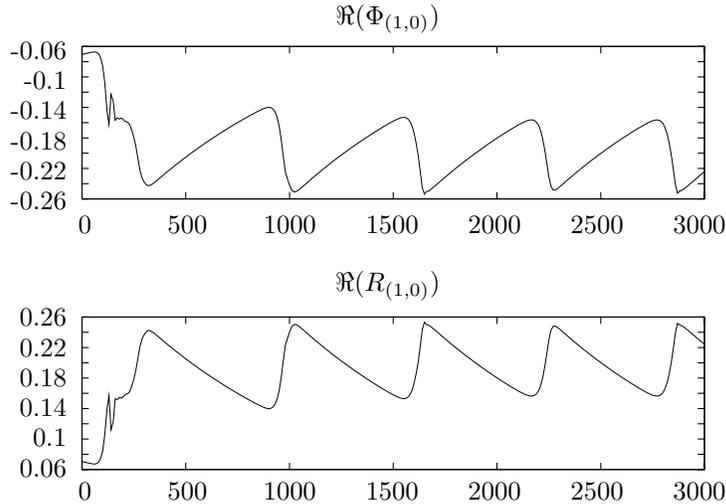}
\end{center}
\caption{Real part of  ${\Phi}_{(1, 0)}$ and ${R}_{(1, 0)}$ for $\beta = 0.001$  and $0 \le t \le 3000$.}\label{fig:zero1D3_10}
\end{figure}  

\section{Conclusion}

Our initial motivation for using the centre manifold theory to study the nonlinear behaviour of Hasegawa--Wakatani equations was to verify the results of our code. The theory proved to be a useful tool in predicting and explaining the behaviour of the results. It also highlighted some  unexpected properties of the Hasegawa-Wakatani equations.

We showed that it is important to carry out the computations on the correct sized grids, otherwise the qualitative behaviour of the solution will be wrong. One advantage of working in the Fourier space is that it is easy to see when a solution is wrong. A noisy solution like the one shown in Figure \ref{fig:beta41_11_small} is a strong indication of numerical errors. The author is not aware of any study that has been carried out to determine what are appropriate sized grids for finite element or finite difference discretisations.

We presented a number of different examples of unstable solutions that arose when using periodic boundary conditions.  Some of our results directly contradict the stable solutions reported in other papers. The experiments presented here suggest that any numerical simulation showing stable solutions should be checked to  verify they are not a consequence of numerical errors or a consequence of terminating the simulations too soon. 

We showed examples of stable bifurcation behaviour occurring when the boundary along the $y$-axis was set to zero.  Our future work will focus on better understanding how the different parameters influence the bifurcation behaviour. We intend to carry out this analysis by once again using the centre manifold analysis, but this time projecting down onto larger subspaces containing more modes of interest, as well as continuing to verify our results by numerical solutions on the full system of equations. 

\bibliographystyle{siam}
\bibliography{nonlinear}

\end{document}